\documentclass[12pt]{amsart}

\usepackage{amsmath,amsthm,amscd,amsfonts,amssymb,epic,eepic,bbm, mathabx}
\usepackage[pagebackref,colorlinks=true,linkcolor=blue,citecolor=blue]{hyperref}
\allowdisplaybreaks

\newtheorem{thm}{Theorem}[section]
\newtheorem{cor}[thm]{Corollary}
\newtheorem{conj}[thm]{Conjecture}

\newtheorem{quest}[thm]{Question}
\theoremstyle{remark}
\newtheorem*{rem}{Remark}

\newcounter{remarkscounter}

\numberwithin{equation}{section}
\newcommand{\A}{\mathbb{A}}
\newcommand{\GL}{\mathrm{GL}}
\newcommand{\SL}{\mathrm{SL}}
\newcommand{\ZZ}{\mathbb{Z}}

\newcommand{\Sym}{\mathrm{Sym}}

\newcommand{\lto}{\longrightarrow}

\newcommand{\CC}{\mathbb{C}}

\newcommand{\sch}{\mathbb{S}}
\newcommand{\la}{\lambda}

\newcommand{\quash}[1]{}

\theoremstyle{definition}
\newtheorem{defn}[thm]{Definition}

\numberwithin{equation}{subsection}

\begin{document}
\title{On tensor third $L$-functions of automorphic representations of $\GL_n(\A_F)$}
\author{Heekyoung Hahn}
\address{Department of Mathematics\\
Duke University\\
Durham, NC 27708}
\email{hahn@math.duke.edu}


%\thanks{The author is thankful for partial support provided by NSF grant DMS-1405708.  Any opinions, findings, and conclusions or recommendations expressed in this material are those of the author and do not necessarily reflect the views of the National Science Foundation.}

\maketitle
\begin{abstract}

Langlands' beyond endoscopy proposal for establishing functoriality motivates interesting and concrete problems in the representation theory of algebraic groups.  We study these problems in a setting related to the Langlands $L$-functions $L(s,\pi,\,\otimes^3),$ where $\pi$ is a cuspidal automorphic representation of $\GL_n(\A_F)$ where $F$ is a global field. 

\end{abstract}
%\tableofcontents

%%%%%%%%%%%%%%%%%%%%%%%%%%%%%%%%%%%%%%%%%%%%%%%%%%%%%%%%%%%%%%%%%%%
%%%%%%%%%%%%%%%%%%%%%%%%%%%%%%%%%%%%%%%%%%%%%%%%%%%%%%%%%%%%%%%%%%%%%%%%%%%
\section{Introduction}

Let $F$ be a number field and let $\A_F$ be adeles of $F$. Let $G$ be a reductive group over $F$. Given a representation $${}^LG\lto \GL_n(\CC),$$ Langlands' functorial conjectures \cite{Langlands_conj} predict there should be a corresponding transfer of automorphic representations of $G(\A_F)$ to automorphic representations of $\GL_n(\A_F)$.

One can ask for a characterization of those automorphic representations in the image.  By the conjectural Langlands correspondence, these should correspond to $L$-parameters
$$
\varphi: \mathcal{L}_F\lto \GL_n(\CC)
$$
such that a $\GL_n(\CC)$-conjugate factors through ${}^LG$, where $\mathcal{L}_F$ is the conjectural Langlands group.  From this optic, one is led to ask how one can detect such parameters.  

Let
${}^{\lambda}G$ be the Zariski closure of $\mathrm{Im}({}^LG)$ viewed as a reductive group over $\CC$.
Then a theorem of Chevalley  states that there exists a representation $\GL_n \lto \GL(V)$ such that ${}^{\lambda}G$ is the stabilizer of a line in $V,$ see \cite{Milne}.  Moreover, one knows via a result of  Larsen and Pink \cite{LP} that if ${}^{\lambda}G$ is connected semisimple, then the dimensions $\dim (V^{{}^{\lambda}G})$ for all representations $V$ of $\GL_n$ characterize ${}^{\lambda}G$ up to conjugation.  Thus the following definition is natural:

\begin{defn}
Let $G$ be an irreducible reductive subgroup of $\GL_n$. We say a representation $r:\GL_n \lto \GL(V_r)$ \textit{detects} $G$ if $G$ fixes a line in $V_r$.
\end{defn}

\begin{rem}
If $G$ is connected  then $r$ detects $G$ if and only if it detects $G^{\mathrm{der}}$.
\end{rem}

\noindent 

The following conjecture is the crux of Langlands' beyond endoscopy proposal \cite{Langlands_beyond}, which aims to prove Langlands functoriality in general:

\begin{conj}\label{conj}
Let $\pi$ be a unitary cuspidal automorphic representation of $\GL_n(\A_F).$ If $\pi$ is a functorial transfer from $G$, then $L(s,\pi,r \otimes \chi)$ has a pole at $s=1$ for some character $\chi \in F^{\times} \backslash \A_F^{\times} \to \CC^\times$ whenever $r$ detects ${}^{\lambda}G.$
\end{conj}

Thus implicit in Langlands proposal is a very concrete question in algebraic group theory.  
\begin{quest}\label{Q}
Given a representation 
$$
r:\GL_n \lto \GL(V_r)
$$
which algebraic subgroups of $\GL_n$ are detected by $r$?
\end{quest}

If $r=\Sym^2$, one knows that every irreducible reductive subgroup of $\GL_n$ detected by $r$ is conjugate to a subgroup of $\mathrm{GO}_n$. Moreover, in this case the Conjecture \ref{conj} is proven by work of Arthur \cite{Arthur}, work of Cogdell, Kim, Piatetski-Shapiro and Shahidi \cite{CKPSS} and work of Ginzburg, Rallis and Soundry \cite{GRS}.  There is a similar statement for $r=\mbox{\Large $\wedge$}^2$.

Apart from these special cases, explicit results are hard to come by.  We mention one case that was discussed in a recent paper of Getz and Klassen \cite{GK}.
Let 
$$
RS: \GL_m\times\GL_m \hookrightarrow \GL_{m^2}
$$
be the representation induced by the usual tensor product. Then it is known that $RS(\GL_m\times\GL_m)$ is detected by $\mathrm{Sym}^m$ (see \cite{GK} for instance). Moreover, the analytic properties of the relevant Langlands  $L$-function is well understood via the Rankin-Selberg theory, although we do not know its automorphicity.
It is important to note that even in this situation Getz and Klassen do not discuss whether there are any other maximal subgroups of $\GL_{m^2}$ that are detected  by $\mathrm{Sym}^m$. 

\begin{rem}
Let $\pi_1$ and $\pi_2$ be cuspidal automorphic representations of $\GL_m(\A_F)$.   Then one knows that
$$
L(s,\pi_1 \times \pi_2,RS)=L(s,\pi_1 \times \pi_2).
$$
Here the function on the left is the Langlands $L$-function and the function on the right is the usual Rankin-Selberg $L$-function.
\end{rem}

In this paper, we examine subrepresentations of the representation 
$$
\otimes^3:\GL_n \lto \GL_{n^3}.
$$
In some sense this is the easiest case to consider after $r=\mathrm{Sym}^2$ and $r=\mbox{\Large$\wedge$}^2$.  This coincides with the setting of \cite{GK} when we take $m=3$ and $n=9$.  In this case we can obtain much more information than what is proven in loc.~cit (see Thereom \ref{thm-GL9-intro} below).

Our first result is a lower bound on $\mathbb{S}_{\la}(\SL_m)$ not to be detected by $\otimes^3$ with respect to the number of nonzero parts of a partition $\la$. Here $\mathbb{S}_{\la}$ is the usual Schur functor associated to partition $\la$.

%%%%%%%%%%%%%%%%%
\begin{thm}\label{thm-Schur-intro}
Let $\lambda$ be a partition with $\ell\leq m$ nonzero parts. If $m> 3\ell$, then the representation $\otimes^3$ does not detect  $\mathbb{S}_{\lambda}(\SL_m)$.
\end{thm}

Theorem \ref{thm-Schur-intro} leads us to ask whether $\mathbb{S}_{\la}(\SL_m)$ is detected by $\otimes^3$ if $\ell\leq m\leq 3\ell$. We consider only the simplest case where $\la=(k)$ (see \S \ref{Schur} for notation). So $\mathbb{S}_{\la}=\Sym^k$. In this case, $\ell=1$ so we ought to study when $m=2$ and $m=3$.

%%%%%%%%%%%%%%%
\begin{thm}\label{thm-A1-intro}
Let 
$$G:=\Sym^{n-1}(\SL_2)\hookrightarrow \GL_n.$$ 
The representation $\Sym^3$ detects $G$
if and only if $n\equiv 1\pmod{4}$.
\end{thm}

Theorem \ref{thm-A1-intro} admits a pleasant partition-theoretic interpretation. Recall the following partition function: For fixed $j, k\in \mathbb{N}$, let $p(k, j, n)$ be the number of partitions of $n$ into at most $j$ parts, with largest part at most $k$. Then we have the following:

\begin{cor}\label{cor-partition-intro}
For any integer $\ell\geq 1$, one has 
\begin{equation}\label{id-1}
p(4\ell-2, 3, 6\ell-3)=p(4\ell-2, 3, 6\ell-4)
\end{equation}and
\begin{equation}\label{id-2}
p(4\ell, 3, 6\ell)-p(4\ell, 3, 6\ell-1)=1.
\end{equation}
\end{cor}

The proof of Corollary \ref{cor-partition-intro} relies on representation theory. It would be interesting to find combinatorial proofs of \eqref{id-1} and \eqref{id-2}. A. J. Yee pointed out the author that one might deduce the identities by employing symmetries of the coefficients of Gaussian polynomials, but the details are not obvious.

For the case $m=3$,  we set
\begin{equation}\label{n}
n=\frac{(k+2)(k+1)}{2},
\end{equation}
for $k\in \mathbb{N}$. Then $ \Sym^k(\SL_3)$ is naturally an irreducible reductive subgroup of $\GL_n$. One has the following result:

%%%%%%%%%%%%%%%%
\begin{thm}\label{thm-A2-intro}  Let $n$ be given as in \eqref{n} and let
$$G:=\Sym^k(\SL_3)\hookrightarrow \GL_n.$$ 
Then for any $k>0$,  the representation $\otimes^3$ detects $G$.
\end{thm}

Finally, we examine closely irreducible connected reductive subgroups of $\GL_9$ that are detected by $\Sym^3$. They are completley classified via the following theorem:

%%%%%%%%%%
\begin{thm}\label{thm-GL9-intro}
Let $G\leq \GL_9$ be an irreducible connected reductive subgroup of $\GL_9$. Then
if the representation $\Sym^3$ detects $G$, then either $\mathrm{Lie}(G^\mathrm{der})=\mathfrak{sl}_2$, or a $\GL_9$-conjugate of $G$ is contained in $RS(\GL_3\times\GL_3)$.
\end{thm}
\begin{rem}
It would be very interesting to investigate non-connected reductive subgroups $G\leq\GL_9$.
\end{rem}

We close the introduction by outlining the paper. In \S \ref{Schur}, we provide an upper bound on $m$ such that $\mathbb{S}_{\la}(\SL_m)$ is detected by the representation $\otimes^3$ in terms of the number of parts of $\la$. In \S \ref{Sym3}, we  discuss when $\Sym^3$ detects $\Sym^k(\SL_2)$. As an application, we present a partition theoretic interpretation of Theorem \ref{thm-A1-intro}. In \S \ref{tensor3}, we prove Theorem \ref{thm-A2-intro} and finally in \S \ref{GL9}, we prove Theorem \ref{thm-GL9-intro}.

%%%%%%%%%%%%%%%%%%%%%%%%%%%%%%%%%%%%%%%%%%%%%%%%%%%%%%%%%%%%%%%
%%%%%%%%%%%%%%%%%%%%%%%%%%%%%%%%%%%%%%%%%%%%%%%%

\section{Schur functors}\label{Schur}

Let $\lambda$ be a partition with at most $m$ parts written as
\begin{equation}\label{partition-mparts}
\lambda=(\lambda_1, \lambda_2, \dots , \lambda_m), \quad \lambda_1\geq\lambda_2\geq \cdots\geq \lambda_m\geq 0,
\end{equation}
and let $|\lambda|$ be the number partitioned by $\lambda$.

For any $m$ dimensional vector space $V$ over $\CC$ and any partition  $\lambda$ of $|\lambda|$ with at most $m$ parts as in \eqref{partition-mparts}, we can apply the Schur functor $\mathbb{S}_{\lambda}$ to $V$ to obtain a representation $\mathbb{S}_{\lambda}(V)$ of $\GL_m$. It remains irreducible when restrict to $\SL_m$. 

By the Littlewood-Richardson formula (compare with \cite[Exercise 15.23]{FH} ), one knows the decomposition of a tensor product of any two irreducible representations of $\mathrm{SL}_m$, namely
\begin{equation}\label{LR}
\mathbb{S}_{\lambda}(V) \otimes \mathbb{S}_{\mu}(V)=\bigoplus_{\nu}C_{\lambda\mu\nu}\mathbb{S}_{\nu}(V).
\end{equation}
Here the coefficient $C_{\lambda\mu\nu}$ are given by the Littlewood-Richardson rule and the sum is over partition $\nu$ of $|\lambda|+|\mu|$.

 Then one has the following result:

%%%%%%%%%%%%%%%%%
\begin{thm}\label{thm-Schur}
Let $\lambda$ be a partition with $\ell\leq m$ nonzero parts. If $m> 3\ell$, then the representation $\otimes^3$ does not detect  $\mathbb{S}_{\lambda}(\SL_m)$.
\end{thm}

\begin{proof}
Let $\lambda$ be a partition of $|\lambda|$ whose number of nonzero parts is $\ell\leq m$. In other words, $\lambda$ can be written as
$$
\lambda=(\lambda_1, \dots, \lambda_{\ell}, \underbrace{0, \dots, 0}_{m-\ell~ \text{times}}),\quad \lambda_1\geq \lambda_2\geq \dots \geq \lambda_{\ell}>0\,.
$$ 
Let $V$ be the standard representation of $\SL_m(\CC)$. Consider
\begin{align}
\mathbb{S}_{\lambda}(V)^{\otimes 3}&\cong \mathrm{Hom}(\mathbb{S}_{\lambda}(V)^{\vee}, \sch_{\la}(V)\otimes \sch_{\la}(V))\nonumber\\
&\cong\bigoplus_{\nu}C_{\lambda\la\nu}~ \mathrm{Hom}(\mathbb{S}_{\lambda}(V)^{\vee}, \sch_{\nu}(V)), \label{Schur-Hom}
\end{align}
where we employ the Littlewood-Richardson formula \eqref{LR} and denote by $\sch_{\la}(V)^{\vee}$ the dual space of $\sch_{\la}(V)$. Therefore the partitions $\nu$ indexing the tensor product of \eqref{Schur-Hom} are partitions of $2|\la|$ with at most $m$ parts.

Let $\lambda^{\vee}$ be the partition such that $\mathbb{S}_{\lambda}(V)^{\vee}=\mathbb{S}_{\lambda^{\vee}}(V).$ Then it is easy to check that $\la^{\vee}$ is
\begin{align} \label{l-shape}
\la^{\vee}=(\underbrace{\la_1, \dots, \la_1}_{m-\ell~ \text{times}}, \la_1-\la_{\ell}, \la_1-\la_{\ell-1}, \dots , \la_1-\la_2, 0)
\end{align}
(compare with \cite[Exercise 15.50]{FH}). Therefore $\la^{\vee}$ must have at least $m-\ell$ nonzero parts.

Notice that two partitions $\mu$ and $\mu'$ with at most $m$ parts determine the same representation of $\mathrm{SL}_m$ if and only if there is an integer $b$ such that
$$
\mu=\mu'+b,
$$
that is, $\mu_i=\mu'_i+b$ for all $1 \leq i \leq m$.  Thus it suffices to prove that if
$m>3\ell$, then
$$
\mathrm{Hom}_{\SL_m}(\mathbb{S}_{\lambda^{\vee}+b}(V), \sch_{\la}(V)\otimes \sch_{\la}(V))=0
$$
for all $b \in \ZZ$ such that $\lambda^{\vee}+b$ is a partition (i.e. all of its entries are nonnegative).  In fact, such a $b$ is necessarily nonnegative by \eqref{l-shape}. Thus partitions of the form $\lambda^{\vee}+b$ always have at least $m-\ell$ nonzero parts.

On the other hand, the Littlewood-Richardson rule forces the partition $\nu$ indexing the decomposition in \eqref{LR} to have at most $2\ell$ nonzero parts. This is a well-known fact, but we will provide a quick proof. To prove this, one first recalls that the number of boxes in the first column of the Young diagram of a given partition denotes the number of nonzero parts of it. Therefore $\la$ will have $\ell$ boxes in the first column in its Young diagram. Now we have to fill this Young diagram of $\la$ by the Littlewood-Richardson rule to obtain the Young diagram for $\nu$. Starting with $\ell$ boxes in the first column corresponding to $\la$, one is allowed to add only as many boxes in the first column as the number of parts of $\la$ which is in turn $\ell$. Therefore the maximum number of boxes in the first column of the Young diagram of $\nu$ is $2\ell$ (see \cite[Appendix A]{FH} for example).

Assuming $m-\ell>2\ell$, we conclude that the partition $\la^{\vee}+b$ can never appear in the decomposition of tensor product, hence we prove that
$$
\mathrm{Hom}_{\mathrm{SL}_m}(\mathbb{S}_{\lambda}(V)^{\vee}, \sch_{\nu}(V))=0.
$$This completes the proof.
\end{proof}

%%%%%%%%%%%%%%%%%%%%%%%%%%%%%%%%%%%%%%%%%%%%%%%%%%%%%%%%%%%%%%%%%%%%%%%%%
%%%%%%%%%%%%%%%%%%%%%%%%%%%%%%%%%%%%%%%%%%%%%%%%%%%%%%%%%%%%%%%%%%%%%%%%%

\section{The subgroup $\Sym^k(\SL_2)$}\label{Sym3}

In this section, we only consider a special types of irreducible reductive subgroups of $\GL_n$, namely $G:=\Sym^k(\SL_2)$. In fact, we provide a precise condition on $k$ for $G$ to be detected by the representation $\Sym^3$, which implies that it is detected by $\otimes^3$.

Recall the Gaussian polynomial
\begin{equation}\label{Gauss}
\left[\begin{array}{c}a\\ b\end{array}\right]_q=\frac{(1-q^a)(1-q^{a-1})\cdots (1-q^{a-b+1})}{(1-q)(1-q^2)\cdots(1-q^b)}
\end{equation}
and the plethysm decomposition for $\GL_2$, namely:

\begin{thm}[Theorem 5.5 \cite{Do}]\label{Pleth}
Let $\dim V=2$. Then there is an isomorphism of $\GL_2$-representations
\begin{equation}\label{Pleth-eqn}
\Sym^j(\Sym^kV)\cong \bigoplus_{w=0}^{\lfloor jk/2\rfloor} (\Sym^{jk-2w}V\otimes {\det}^{jk-w})^{\oplus N(j, k, w)}, 
\end{equation}
where $N(j, k, w)$ is the coefficient of $q^w$ in the polynomial $(1-q)\left[\begin{array}{c}j+k\\ k\end{array}\right]_q.$
\end{thm}

In particular if $j=2$ in Theorem \ref{Pleth} and $V$ is a vector space of characteristic zero such that $\dim V=2$, then one has
\begin{equation}\label{free}
\Sym^2(\Sym^kV)\cong \bigoplus_{w\geq 0}^{\lfloor k/2\rfloor}\Sym^{2k-4w}V
\end{equation}
as $\SL_2$-representations (see \cite[Exercise 5.16]{Do}, for instance). One should notice that some terms in \eqref{Pleth-eqn} vanishes in the decomposition in \eqref{free}. In the following theorem, we give a complete answer to the question of when the representation $\Sym^3$ detects the subgroup $\Sym^{n-1}(\SL_2)$ of $\GL_n$:

%%%%%%%%%%%%%%%%%%%%%%%%
\begin{thm}\label{A1-thm}
Let 
$$G:=\Sym^{n-1}(\SL_2)\hookrightarrow \GL_n.$$ 
The representation $\Sym^3$ detects $G$
if and only if $n\equiv 1\pmod{4}$.
\end{thm}

\begin{proof}
Let $W$ be a representation of $\SL_2$. Then it is well-known that one has an isomorphism of $\SL_2$ representations
$$
W^{\otimes 3}\cong \mathrm{Hom}( W^{\vee}, W\otimes W)=\mathrm{Hom}(W^{\vee}, \Sym^2W)\oplus\mathrm{Hom}( W^{\vee}, \mbox{\Large $\wedge$}^2W),
$$
where $W^{\vee}$ denotes the dual space of $W$. Considering the highest weights, it is easy to see that
$$
\Sym^3W\hookrightarrow \mathrm{Hom}( W^{\vee}, \Sym^2W)
$$
as $\SL_2$-representations.
Let $W=\Sym^kV$, where $V$ is a complex vector space with $\dim V=2$.  Then
\begin{align*}
\mathrm{Hom}( W^{\vee}, \Sym^2W)&=\mathrm{Hom}((\Sym^kV)^{\vee},\Sym^2(\Sym^kV))\\
&=\mathrm{Hom}(\Sym^k(V^{\vee}), \Sym^2(\Sym^kV))\\
&=\mathrm{Hom}(\Sym^k V, \Sym^2(\Sym^kV))\\
&\cong \bigoplus_{i=0}^{\lfloor k/2\rfloor}\mathrm{Hom}( \Sym^kV,\Sym^{2k-4i}V),
\end{align*}
where in the last equality we use the Plethysm decomposition formula \eqref{free} and the fact that $V^{\vee}\cong V$ for $\dim V=2$. By Shur's lemma, 
\begin{equation}\label{hom}
\mathrm{Hom}_{\SL_2}(\Sym^kV, \Sym^{2k-4i}V)\leq 1,
\end{equation} and 
 $\mathrm{Hom}_{\SL_2}(\Sym^kV, \Sym^{2k-4i}V)=1$ if and only if $k=2k-4i$. Therefore $k\equiv 0 \pmod{4}$. By taking $k=n-1$, we obtain the result.
\end{proof}

One then obtains immediate partition theoretic interpretation of Theorem \ref{A1-thm}. It is well-known that the generating function of $p(k, j, n)$ is the Gaussian polynomial
\begin{equation}\label{partition-eqn}
\sum_{n\geq 0}p(k, j, n)q^n=\left[\begin{array}{c}j+k\\ k\end{array}\right]_q
\end{equation}(see \cite[Proposition 1.7.3]{Stan}, for instance).  We deduce the following corollary:

%%%%%%%%
\begin{cor}
For any integer $\ell\geq 1$, one has 
$$p(4\ell, 3, 6\ell)-p(4\ell, 3, 6\ell-1)=1$$and
$$p(4\ell-2, 3, 6\ell-3)=p(4\ell-2, 3, 6\ell-4).$$
\end{cor}

\begin{proof}
Note that when $3k-2w=0$ in the decomposition \eqref{Pleth-eqn}, one has only one $1$-dimensional summand. We then use the fact \eqref{hom} in the proof of Theorem \ref{A1-thm} to deduce that
$$
N(3, k, 3k/2)=\begin{cases} 1 &\mbox{if } k \equiv 0 \pmod{4},\\
0& \mbox{if } k\equiv 2 \pmod{4}.
\end{cases}
$$
On the other hand, by Theorem \ref{Pleth}, since $N(3, k, 3k/2)$ denotes the coefficient of $q^{3k/2}$ in the polynomial $(1-q)\left[\begin{array}{c}3+k\\ k\end{array}\right]_q$,  $N(3, k, 3k/2)$ is equal to the coefficient of $q^{3k/2}$ in $\left[\begin{array}{c}3+k\\ k\end{array}\right]_q$ minus the coefficient of $q^{3k/2-1}$ in $\left[\begin{array}{c}3+k\\ k\end{array}\right]_q$. Using the partition interpretation of the Gausssian polynomial as in \eqref{partition-eqn}, we derive that
$$
p(k, 3, 3k/2)-p(k, 3, 3k/2-1)=\begin{cases} 1 &\mbox{if } k \equiv 0 \pmod{4}\\
0& \mbox{if } k\equiv 2 \pmod{4},
\end{cases}
$$which in turn completes the proof.
\end{proof}

%%%%%%%%%%%%%%%%%%%%%%%%%%%%%%%%%%%%%%%%%%%%%%%%%%%%%%%%%%%%%%%%%%%%%%%%%%%%%%%
%%%%%%%%%%%%%%%%%%%%%%%%%%%%%%%%%%%%
\section{The subgroup $\Sym^k(\SL_3)$}\label{tensor3}

In this section, we closely study the reductive subgroups $\Sym^k(\SL_3)$ for  $k>0$.

%%%%%%%%%%%%%%%%%%%%%%%%%%%%%%%%%%%%%%%%%%%
\begin{thm}\label{An-thm}
Let $n=\frac{(k+2)(k+1)}{2}$ and let
$$G:=\Sym^k(\SL_3)\hookrightarrow \GL_n$$ 
for $k>0$. Then  the representation $\otimes^3$ detects $G$.
\end{thm}

\begin{proof}
Let $V$ be the standard representation of $\SL_3$. For $k>0$, we consider
\begin{equation}\label{nontivial-Hom}
(\Sym^kV)^{\otimes 3}\cong \mathrm{Hom}((\Sym^kV)^{\vee}, \Sym^kV\otimes \Sym^kV),
\end{equation}
where $(\Sym^kV)^{\vee}$ denotes the dual space of $\Sym^kV$. We wish to investigate to see if $$
\mathrm{Hom}((\Sym^kV)^{\vee}, \Sym^kV\otimes \Sym^kV)\neq 0.
$$ 

As a special case of Pieri's formula (compare with \cite[Exercise 15.33]{FH} ), one has that
\begin{equation}\label{S}
\Sym^kV \otimes \Sym^k V=\bigoplus_{i=0}^k\mathbb{S}_{(k+i, k-i)}(V).
\end{equation}
Let $\lambda$ (resp. $\lambda^{\vee}$) be the partition corresponding to $\Sym^kV$ (resp. $(\Sym^kV)^{\vee}$). It is easy to check that if $m=3$, we have $\lambda^{\vee}=(k,k,0)$ (see \cite[Excercise 15.50]{FH} for example). In particular, $\lambda^{\vee}$ appears in the decomposition of \eqref{S} if one takes $i=0$.
\end{proof}

%%%%%%%%%%%%%%%%%%%%%%%%%%%%%%%%%%%%%%%%%%%%%%%%%%%%%%%%%%%%%%%%%
%%%%%%%%%%%%%%%%%%%%%%%%%%%%%%%%%%%%%%%%%%%%%%%%
\section{Connected reductive subgroups of $\GL_9$ detected by $\Sym^3$}\label{GL9}

Throughout this section, we assume that $G\leq \GL_9$ is an irreducible connected reductive subgroup. We characterize those $G$ detected by $\Sym^3$:

%%%%%%%%%%
\begin{thm}\label{thm-GL9}
Let $G\leq \GL_9$ be an irreducible connected reductive subgroup.
If the representation $\Sym^3$ detects $G$, then either $\mathrm{Lie}(G^\mathrm{der})=\mathfrak{sl}_2$, or a $\GL_9$-conjugate of $G$ is contained in $RS(\GL_3\times\GL_3)$.
\end{thm}

\begin{proof}
Let
$$
\rho : G\hookrightarrow \GL_9
$$
be the natural representation. Recall that $G=Z_GG^\mathrm{der}$, where $Z_G$ is the center of $G$. Thus it suffices to consider the irreducible representation of $G^{\mathrm{der}}$  with dimension $9$.

From the dimension table of irreducible representations in \cite{FK}, one knows that $\mathrm{Lie}(G^{\mathrm{der}})$ is either $\mathfrak{sl}_2$, $\mathfrak{sl}_3\times \mathfrak{sl}_3$, $\mathfrak{so}_3\times \mathfrak{so}_3$ or $\mathfrak{so}_9$. 

The case $\mathrm{Lie}(G^{\mathrm{der}})=\mathfrak{sl}_2$ coincides with Theorem \ref{A1-thm} taking $n=9$. In the case where $\mathrm{Lie}(G^{\mathrm{der}})$ is isomorphic to $\mathfrak{sl}_3\times \mathfrak{sl}_3$ or $\mathfrak{so}_3\times \mathfrak{so}_3$, it  is clear that $\mathrm{Lie}(G^{\mathrm{der}})$ is contained in a $\GL_9$-conjugate of $\mathrm{Lie}(RS(\GL_3 \times \GL_3))$.  It follows that $G^{\mathrm{der}} \leq RS(\GL_3 \times \GL_3)$.  On the other hand, since the representation $\rho$ is irreducible, $\rho(Z_G)$ must be contained in the subgroup of scalar matrices, and hence $G \leq RS(\GL_3 \times \GL_3)$.
For the last case $\mathrm{Lie}(G^{\mathrm{der}})=\mathfrak{so}_9$, we let $V$ be the standard representation of $\mathfrak{so}_9$. Then from the decomposition of $\Sym^3V$ into the direct sum of Schur functors \cite[Excercise 19.21(ii)]{FH}, one knows that there is no trivial representation occurring in the decomposition. This completes our proof.
\end{proof} 
 
Lastly we want to note that it would be interesting to investigate irreducible non-connected reductive subgroups of $\GL_9$, since the rank is not so big. This would require finite group theory and Clifford theory.

\section*{Acknowledgements}

The author is grateful to J. R. Getz for suggesting this project and to Leslie Saper for answering various questions on representation theory. The author also thanks A. J. Yee for a useful conversation on the Gaussian polynomials.

% ----------------------------------------------------------------

\end{document}